\documentclass[12pt]{amsart}
\pdfoutput=1
\usepackage[parfill]{parskip}    
\usepackage{amsfonts}
\usepackage{amsmath}
\usepackage{amsthm}
\usepackage{epstopdf}
\usepackage{hyperref}
\usepackage[all]{xy}
\usepackage[pdftex]{color,graphicx}
\usepackage{psfrag}
\usepackage{times}

\title[Tight contact structures on laminar free hyperbolic three-manifolds]{
Tight contact structures on \\ laminar free hyperbolic three-manifolds}
\author{Tolga Etg\"u}
\address{Department of Mathematics, Ko\c{c} University, \.Istanbul 34450 TURKEY \newline and \newline Mathematical Sciences Research Institute, Berkeley, CA 94720 USA}
\email{tetgu@ku.edu.tr}
\date{\today}

\newtheorem{theorem}{Theorem}

\newtheorem{proposition}[theorem]{Proposition}

\newtheorem*{thm*}{Theorem} \theoremstyle{definition}
\newtheorem{remark}[theorem]{Remark}
\newcommand{\M}{M(m; p/q)}
\newcommand{\Mm}{M(m; 1/1)}
\newcommand{\Z}{\mathbb{Z}}
\newcommand{\Q}{\mathbb{Q}}

\begin{document}

\begin{abstract}

Whether every hyperbolic $3$--manifold admits a tight contact structure or not is an open question. Many hyperbolic $3$--manifolds contain taut foliations and taut foliations can be perturbed to tight contact structures. The first examples of hyperbolic $3$--manifolds without taut foliations were constructed by Roberts, Shareshian, and Stein \cite{RSS}, and infinitely many of them do not even admit essential laminations as shown by Fenley \cite{F}. In this paper, we construct tight contact structures on a family of three-manifolds including these examples. These contact structures are described by contact surgery diagrams and their tightness is proved using the contact invariant in Heegaard Floer homology. 

\end{abstract}

\maketitle

\section{Introduction}

A differential $1$--form $\alpha$ on an oriented $3$--manifold is called a (positive) {\sl contact form} and the $2$--plane field given as $\ker \alpha$ is called a (co-oriented, positive) {\sl contact structure} if $\alpha \wedge d \alpha > 0$. A contact structure $\xi$ on $M$ is called {\sl overtwisted} if there is a disk $D$ embedded in $M$ such that the tangent plane $T_xD$ to $D$ is the same as $\xi_x$ for every  $x \in \partial D$, otherwise it is called {\sl tight}. Every closed, oriented $3$--manifold carries an overtwisted  contact structure \cite{be, e, m}, but some of them admit no tight contact structure \cite{eh}. This is a manifestation of the fact that tight contact structures carry more information on the topology of the underlying manifold. The existence problem of tight contact structures on $3$--manifolds is still open. 

As a result of the verification of Thurston's Geometrization Conjecture by Perelman \cite{p1, p2, p3} (also see \cite{MT}) we now know that any closed $3$--manifold can be decomposed into geometric pieces along essential spheres and incompressible tori. Tight contact structures on a connected sum are obtained by connected sum of tight contact structures on the summands \cite{col}, and $3$--manifolds with incompressible tori admit infinitely many tight contact structures \cite{col2, hkm}. Hence to solve the aforementioned existence problem, it suffices to consider closed, oriented $3$--manifolds with geometric structures. This problem is completely resolved for Seifert fibered spaces. More precisely, a closed, oriented Seifert fibered space carries a tight contact structure if and only if it is not  obtained by $(2n-1)$-surgery along the $(2,2n+1)$ torus knot in $S^3$ for $n \in \Z_+$ \cite{ls}. On the other hand, a taut foliation on a $3$--manifold can be perturbed to a (fillable and universally) tight contact structure \cite{ET}, and many hyperbolic $3$--manifolds are known to carry taut foliations by a fundamental theorem of Gabai \cite{G}, in fact, it was once conjectured that every hyperbolic $3$--manifold does so. The first examples of closed, hyperbolic $3$--manifolds without any taut foliation were given in \cite{RSS}. As proved by Fenley \cite{F}, infinitely many of these examples and certain closely related $3$--manifolds do not even have essential laminations. Fenley's examples are apparently the only known examples of laminar free closed hyperbolic $3$--manifolds. Since the tight contact structures induced by taut foliations are universally tight and fillable, a more general structure, such as an essential lamination, might have been expected to be the source for virtually overtwisted and non-fillable tight contact structures on hyperbolic $3$--manifolds. It is a natural question to ask whether laminar free hyperbolic $3$--manifolds carry tight contact structures or not. 

In this paper, we construct tight contact structures on a general family of $3$--manifolds which contain those without taut foliations given in \cite{RSS} and Fenley's laminar free examples \cite{F}. In particular, we prove the following 
\begin{theorem}
There exist infinitely many closed hyperbolic $3$--manifolds which contain no essential lamination but admit tight contact structures.
\end{theorem}

\subsection*{Acknowledgments} We would like to thank Bar\i\c{s} Co\c{s}kun\"uzer for bringing to our attention Fenley's work \cite{F}, Rachel Roberts for helpful conversations. We also appreciate Bar\i\c{s} Co\c{s}kun\"uzer and Andr\'as Stipsicz's comments on a draft of this paper. This research was partially supported by T\"UB\.ITAK, the Scientific and Technological Research Council of Turkey.

\section{Hyperbolic $3$-manifolds with no essential lamination}

As it is discussed above, we will focus on hyperbolic $3$--manifolds since the existence problem of tight contact structures is completely resolved for Seifert fibered $3$--manifolds \cite{ls}.  Taut foliations and essential laminations are two important notions widely used to study hyperbolic $3$--manifolds. 
A codimension--$1$ foliation on a compact, connected manifold is called {\sl taut} if there is a circle transversely intersection every leaf. A compact $3$--manifold which admits a taut foliation has infinite fundamental group \cite{nov} and is irreducible unless it is covered by $S^1 \times S^2$ \cite{pal}. Conversely, many hyperbolic $3$--manifolds carry taut foliations \cite{G}. (See \cite{cale} for more on taut foliations and essential laminations in $3$--manifolds.)

Since taut foliations can be perturbed to tight contact structures, we are mainly interested in hyperbolic $3$--manifolds without taut foliations. 
The first such examples were given by Roberts, Shareshian and Stein \cite{RSS}. These examples are among the $3$--manifolds obtained by performing $p/q$-surgery along a section of the torus bundle over the circle with monodromy $\phi_m$ given by matrices of the form 
$$ A_m =
\left[
\begin{array}{cc}
  m &  1     \\
  -1 & 0      
\end{array}
\right] \in SL(2;\Z) \ ,
$$
where $m,p,q \in \Z$ with $(p,q)$=1 and $q\geq 0$.  We denote the result of this surgery by $\M$. For a fixed $m \leq -3$, all but finitely many of $\M$ are hyperbolic by the hyperbolic surgery theorem of Thurston \cite{th}. The main result of \cite{RSS} is 

\begin{theorem}[Roberts-Shareshian-Stein]\label{rss} If $m\leq -3$ is odd, $p$ is odd, and $p \geq q$, then $\M$ contains no taut foliation. 
\end{theorem}

A generalization of taut foliations of $3$--manifolds is the notion of an essential lamination introduced by Gabai and Oertel \cite{GO}. A {\sl lamination} is a foliation of a closed subset of the manifold. A lamination is said to be {\sl essential} if it contains no sphere leaf, no torus leaf bounding a solid torus, has an irreducible complement, and the leaves in the boundary of its complement are incompressible and end incompressible in the corresponding components of the complement. 
In the absence of a taut foliation in a $3$--manifold, it is still possible to find an essential lamination, in fact many hyperbolic $3$--manifolds without taut foliations mentioned in Theorem~\ref{rss} do contain essential laminations \cite{RSS}. On the other hand, Fenley proved the existence of infinitely many closed hyperbolic $3$--manifolds with no essential lamination.

\begin{theorem}[Fenley \cite{F}]\label{lam} If $m \leq -4$ and $|p-2q|=1$, then $\M$ contains no essential lamination. 
\end{theorem}

The manifold $\M$ can alternatively be described as the result of a Dehn surgery along the binding of an open book decomposition with punctured torus pages. Using this description, we obtain a
$4$-manifold (see Figure~\ref{openbook}) whose boundary is given by such an open book decomposition with monodromy $\phi_m$, and Figure~\ref{surgery} is the result of $p/q$-surgery along the binding of this open book. Here, we use the decomposition of $\phi_m$ into (right-handed) Dehn twists as $\phi_m=\tau_a^{-m+1}\tau_b\tau_a$, where  $\tau_a$ and $\tau_b$ are  Dehn twists along the standard generators $a$ and $b$ of the first homology of a punctured torus. This can be seen from the fact that, with respect to the basis $\{ a, b \}$, $\tau_a$ and $\tau_b$ can be given by the matrices
$$ \left[
\begin{array}{cc}
  1 &  1     \\
  0 & 1      
\end{array}
\right]  
\ \mbox{ and } \ 
 \left[
\begin{array}{cc}
  1 &  0     \\
  -1 & 1      
\end{array}
\right] \ , 
$$
respectively. Since we can freely change the monodromy by conjugation, we will sometimes work with $\tau_a^{-m+2}\tau_b$ instead of $\tau_a^{-m+1}\tau_b\tau_a$ for convenience. 

\begin{figure}[!ht]
\centering
\includegraphics[scale=0.5]{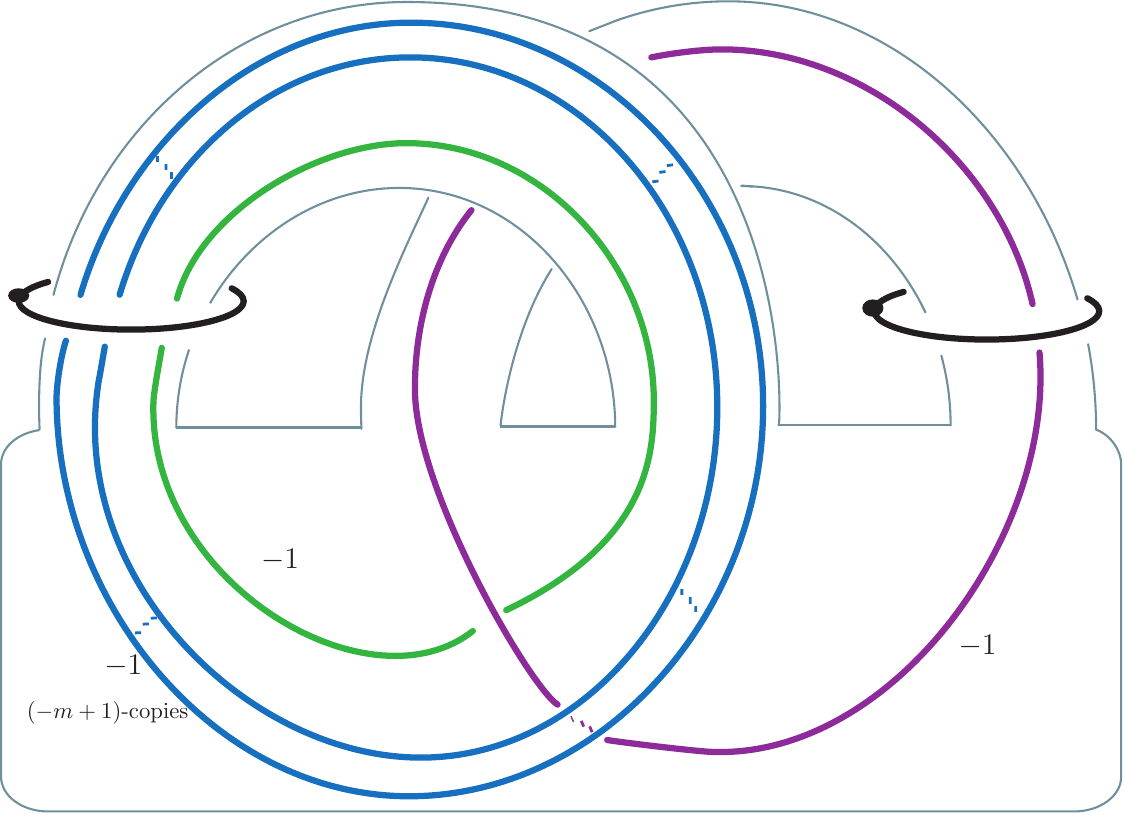} \caption{Handlebody decomposition of a $4$-manifold $X_m$ whose boundary has an open book decomposition with punctured torus page and monodromy $\phi_m=\tau_a^{-m+1}\tau_b\tau_a$ with $m < 1$.}
\label{openbook}
\end{figure}
\begin{figure}[!ht]
\centering
\includegraphics[scale=0.5]{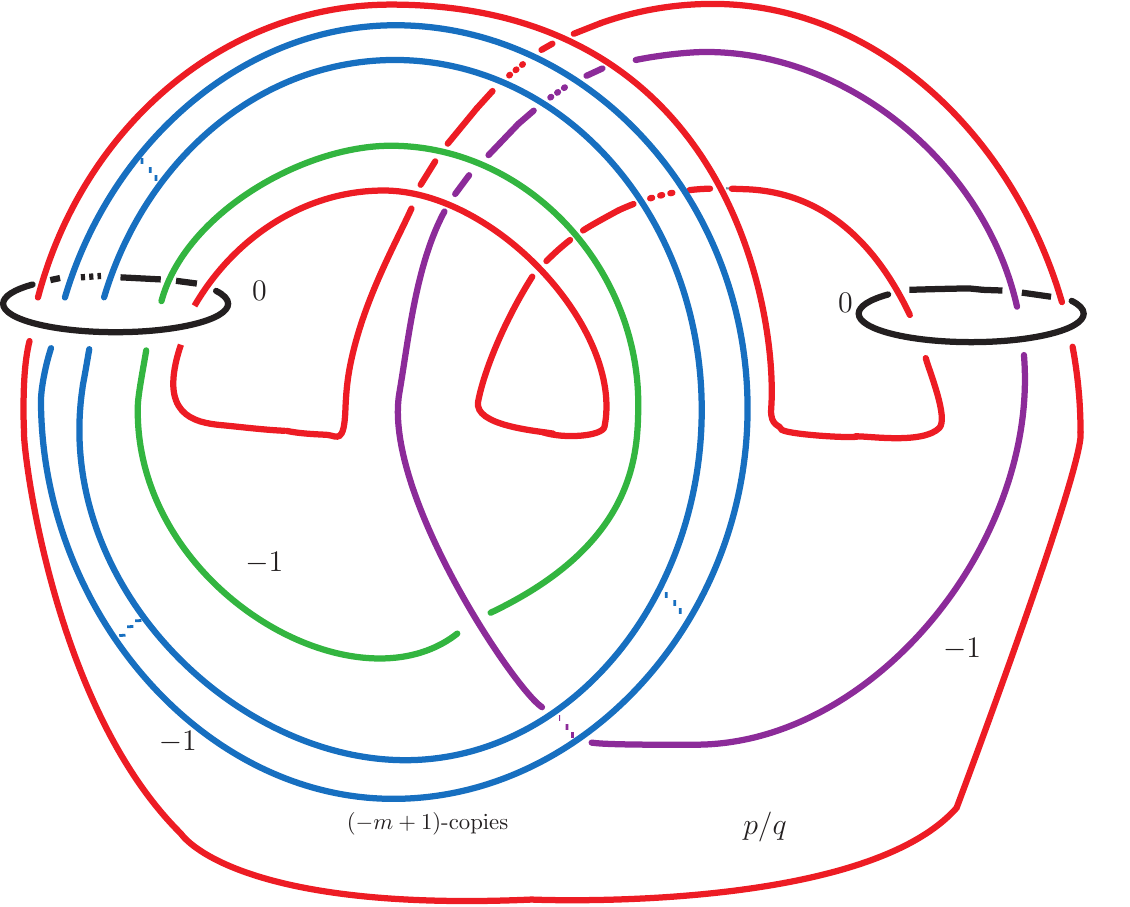} \caption{A surgery diagram for the $3$--manifold $\M$ for $m<1$.}
\label{surgery}
\end{figure}

\section{Contact Structures and Heegaard Floer Invariants}

In this section, we first describe contact structures on $\M$ by giving contact surgery diagrams. (For the basics of contact surgery diagrams see e.g. \cite{DG, DGS}.) Then we prove that these are tight whenever $p/q \neq 1$. If $p/q < 1$, then these diagrams involve only contact surgeries of negative framing, hence the contact structures they describe can alternatively be obtained by contact surgery diagrams which involve only $-1$ contact framing, and therefore these contact structures are in fact Stein fillable (in particular, tight). When $p/q>1$ tightness of the contact structures are proved by showing that they have nonzero Heegaard Floer invariants \cite{OS}.

\begin{proposition}\label{main}
For every $m \in \Z_{\leq 0}$ and $p/q \in \Q$, $\M$ admits a tight contact structure. 
\end{proposition}
\begin{proof}
The contact surgery diagram in Figure~\ref{contact} describes contact structures on $\M$ for $p/q \neq 1$. To see that the underlying manifold is $\M$, consider the Legendrian surgery diagram in Figure~\ref{torus} (right). Figure~\ref{contact} is obtained by cancelling the $1$-handles by using a $-1$-framed $2$-handle for each. Note that, unless $|p-q|=1$, such a diagram does not uniquely determine a contact structure, but nonetheless the contact structures it describes are related in a well-understood way and the discussion below applies to all.
\begin{figure}[!ht]
\centering
\includegraphics[scale=0.7]{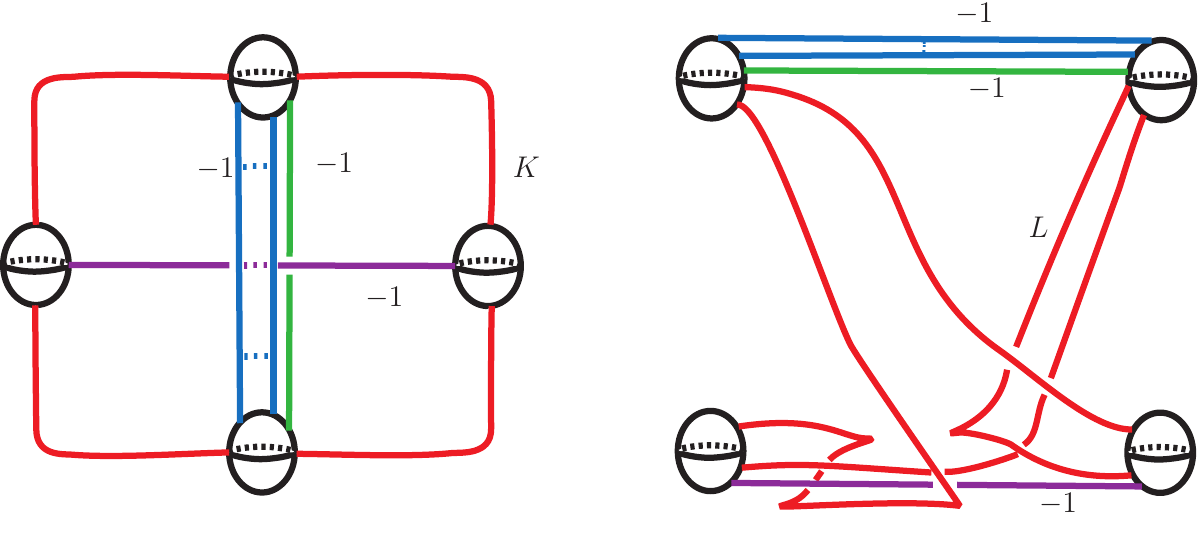} \caption{{\sl Left:} An alternative way to draw the smooth surgery diagram of $X_m$ in Figure~\ref{openbook}. $K$ indicates the curve along which $p/q$-surgery is performed to get $\M$ from $\partial X_m$.  {\sl Right:} A Legendrian surgery description of a Stein structure on $X_m$. $L$ indicates a Legendrian realization of $K$ with $tb(L)=1$. (Note that the fact that $\tau_a^{-m+2}\tau_b$ and $\tau_a^{-m+1}\tau_b\tau_a$ are conjugate manifests itself while passing from left to right.)}
\label{torus}
\end{figure}
\begin{figure}[!ht]
\centering
\includegraphics[scale=0.5]{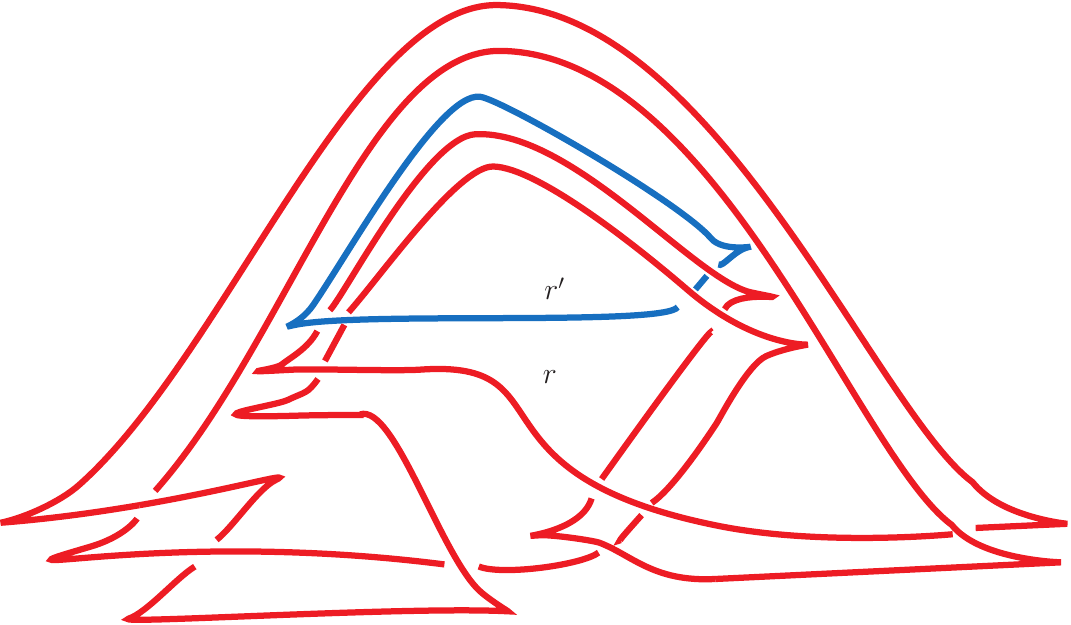} \caption{Surgery diagram for contact structures on $\M$, where $r=(p-q)/q$ and $r'=-1/(-m+1)$.}
\label{contact}
\end{figure}

As it is explained in \cite{DGS}, any contact surgery on a Legendrian knot with a negative contact framing can be realized as a sequence of contact $-1$ surgeries along Legendrian knots. Therefore, when $r=(p-q)/q< 0$, any contact structure $\xi$ given by Figure~\ref{contact} has another contact surgery description which involves only $-1$ contact surgeries. As a consequence, this contact structure is Stein fillable, hence tight.

Now assume that $r>0$. In order to prove the tightness of the contact structure $\xi$ given by the contact surgery diagram in Figure~\ref{contact}, we use the contact invariant in Heegaard Floer homology. Since the contact invariant of an overtwisted contact structure vanishes \cite{OS}, it suffices to show that the contact invariant of $\xi$ is nonzero. It is also known that the contact invariant of the result of a contact surgery with negative framing on a Legendrian knot in a contact manifold with nonvanishing contact invariant has nonzero contact invariant \cite{LS}. On the other hand, the Legendrian knot $L$ with contact framing  $r$ in Figure~\ref{contact} has Thurston-Bennequin invariant $tb(L)=1$, its smooth type is (the positive) $5_2$ knot $K$ which has $4$-ball genus $g_4(K)=1$. It is known that if a Legendrian knot in $S^3$ with the standard contact structure satisfies $tb(L)=2g_4(K) -1$, then any contact surgery along $L$ with positive contact framing results in a contact structure with nonzero contact invariant \cite{LS2}. 

Contact surgery with $r=0$ framing is not well-defined since all the contact structures on the filling solid torus which can be glued to the complement of a standard neighborhood of the surgery curve  are overtwisted. So $p/q=1$ case has to be treated separately. Apparently, the underlying manifold $\Mm$ is a Seifert fibered space. In fact, it can be obtained by $(-m+2)$-surgery on the right-handed trefoil as demonstrated in Figure~\ref{whitehead}.
\begin{figure}[!ht]
\centering
\includegraphics[scale=0.75]{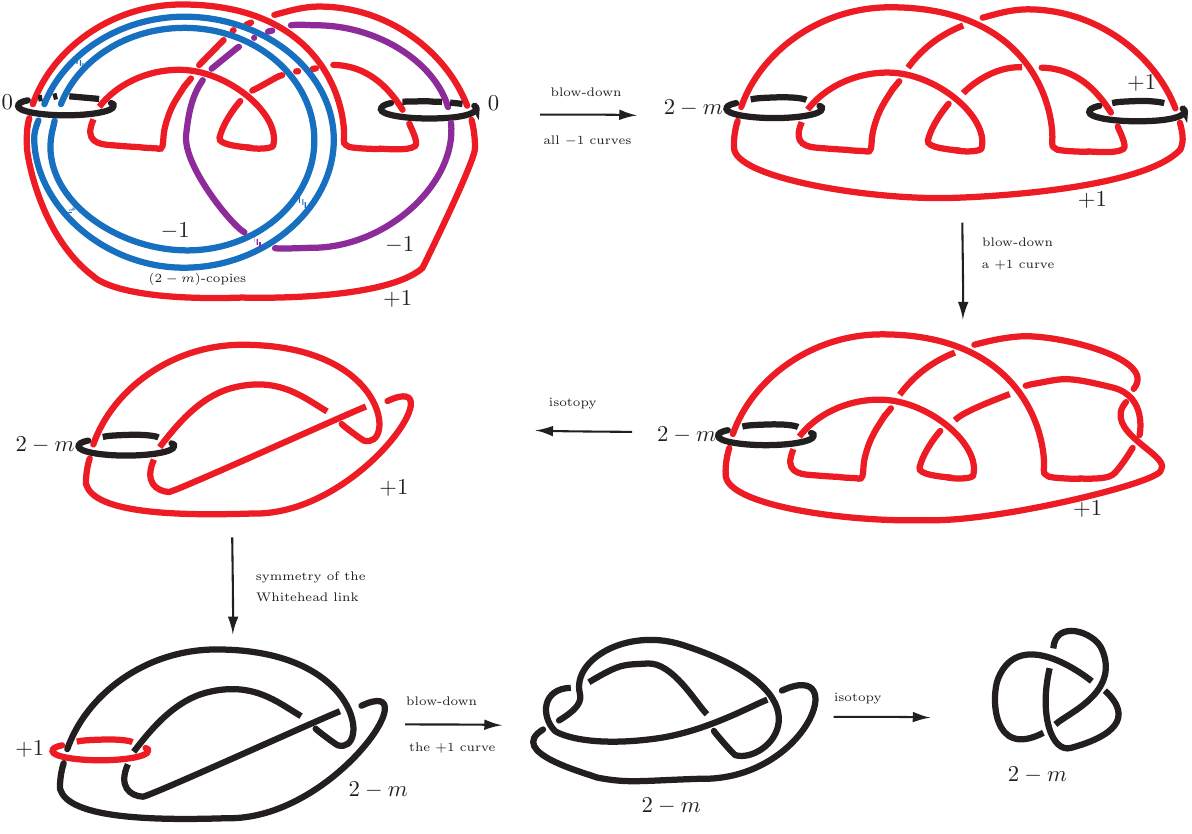} \caption{$\Mm$ can be obtained by a surgery on the right-handed trefoil.}
\label{whitehead}
\end{figure}
Since the maximal Thurston-Bennequin number and the $4$-ball genus of the right-handed trefoil are both $1$, the main result of \cite{LS2} implies that $\Mm$ admits a tight contact structure.  
\end{proof}

\begin{remark}
The manipulation of the surgery diagrams given in the first two rows of Figure~\ref{whitehead} demonstrates that $M(-3;5/2)$ is the Weeks manifold since it can be obtained by $(5,5/2)$-surgery on the Whitehead link. Note that tight contact structures on the Weeks manifold were previously constructed by Stipsicz \cite{S}.
\end{remark}

\end{document}